\documentclass[12pt]{amsart}
\pdfoutput=1
\usepackage{latexsym}
\usepackage[centertags]{amsmath}
\usepackage{amsfonts}
\usepackage{amssymb}
\usepackage{amsthm}
\usepackage{newlfont}
\usepackage{graphics}
\usepackage{color}

\usepackage{wrapfig}

\usepackage[usenames,dvipsnames]{xcolor}
\usepackage{graphicx}

\usepackage{wrapfig}

\newtheorem{theo}{Theorem}

\newtheorem{lem} [theo]{Lemma}

\newtheorem{prop}[theo]{Proposition}

\newtheorem{conj}[theo]{Conjecture}
\newtheorem{algor}[theo]{Algorithm}

\makeatletter \@addtoreset{equation}{section}
\@addtoreset{theo}{}\makeatother

\def\bD{\overline{D}}

\newcommand{\area}{\operatorname{\texttt{area}}}

\newcommand{\dinv}{\operatorname{\texttt{dinv}}}
\newcommand{\bounce}{\operatorname{\texttt{bounce}}}

\def\k{\vec{k}}
\def\r{\dot{r}}%ÖÕÖ¹ÖÈ
\def\S{\mathbb{S}}%´¦ÀíµÄºìÉ«¼ýÍ·
\def\d{\mathbb{S'}}%ÐéÏߵĺìÉ«¼ýÍ·

\def\W{\mathbb{W}}

\def\L{\mathcal{L}}%×ó²à
\def\R{\mathcal{R}}%ÓÒ²à
\def\U{\mathbf{U}}%Æð'Öȼ¯ºÏ
\def\m{<^s}%sweep map ÏÂС.

\def\SW{\texttt{SW}}
\def \CD {{\cal D}}
\def \TAU {{\cal T}}
\def\oD{\overline{D}}

\title{Dinv, Area, and Bounce for $\k$-Dyck paths}

\author{Guoce Xin$^{1,*}$ and Yingrui Zhang$^2$}

 \address{ $^{1,2}$School of Mathematical Sciences, Capital Normal University,
 Beijing 100048, PR China}
 \email{$^1$\texttt{guoce\_xin@163.com}\ \& $^2$\texttt{zyrzuhe@126.com}}
\date{November 9, 2020}

\begin{document}

\begin{abstract}
The well-known $q,t$-Catalan sequence has two combinatorial interpretations as weighted sums of ordinary Dyck paths: one is Haglund's area-bounce formula, and the other is
Haiman's dinv-area formula. The zeta map was constructed to connect these two formulas: it is a bijection from ordinary Dyck paths to themselves, and it takes dinv to area,
and area to bounce. Such a result was extended for $k$-Dyck paths by Loehr. The zeta map was extended by Armstrong-Loehr-Warrington for a very general class of paths.

In this paper, We extend the dinv-area-bounce result for $\k$-Dyck paths by: i) giving a geometric construction for the bounce statistic of a $\k$-Dyck path, which includes
the $k$-Dyck paths and ordinary Dyck paths as special cases; ii) giving
a geometric interpretation of the dinv statistic of a $\k$-Dyck path. Our bounce construction is inspired by Loehr's construction and Xin-Zhang's linear algorithm for
inverting the sweep map on $\k$-Dyck paths. Our dinv interpretation is inspired by Garsia-Xin's visual proof of dinv-to-area result on rational Dyck paths.
\end{abstract}

\maketitle

\noindent
\begin{small}
 \emph{Mathematic subject classification}: Primary 05A19; Secondary 05E99.
\end{small}
%\subjclass[2010]{Primary 15A15; Secondary 05A15, 11B83}

\noindent
\begin{small}
\emph{Keywords}: $q,t$-Catalan numbers; sweep map; $\k$-Dyck paths.
\end{small}

\section{Introduction}
In their study of the space $\mathcal{DH}_n$ of diagonal harmonics \cite{qt-Catalan}, Garsia and Haiman introduced a $q, t$-analogue of the Catalan numbers, which they
called the $q, t$-Catalan sequence. There are several equivalent characterizations of the (original) $q, t$-Catalan sequence, which includes two combinatorial
formulas as weighted sums over the set $\mathcal{D}_n$ of Dyck paths of length $n$: One is
Haiman's dinv-area $q,t$-Catalan sequence
$$HC_n(q, t) =\sum_{D\in \mathcal{D}_n}q^{\dinv(D)}t^{\area(D)} (n = 1, 2, 3,...);$$
the other is
Haglund's area-bounce $q,t$-Catalan sequence
$$C_n(q, t) = \sum_{D \in \mathcal{D}_n} q^{\area(D)} t^{\bounce(D)} (n = 1, 2, 3,...).$$
They are connected by a bijection, called the zeta map $\zeta$, from $\mathcal{D}_n$ to itself \cite[Section~5]{GCCLP-b-ideals,Haglund-bounce}. More precisely, we have the
bi--statistic equality:
$$(\area(\zeta(D)), \bounce(\zeta(D)) = (\dinv(D), \area(D)).$$
There are many interesting results
and generalizations related to the $q,t$-Catalan numbers $C_n(q,t)$. The $q,t$ symmetry $C_n(q,t)=C_n(t,q)$ is proved as a consequence of
the well-known Shuffle theorem of Carlsson and Mellit \cite{proofshuffle}, and its generalization for rational $q,t$ Catalan numbers is proved as a consequence of the
rational Shuffle theorem of Mellit \cite{Mellit}. But combinatorially proving these $q,t$ symmetry properties has been intractable.

D. Armstrong, N. Loehr, and G. Warrington \cite[Section 3.4]{sweepmap} introduced the sweep map for a very general class of paths,
including the zeta map for Dyck paths and rational Dyck paths as special cases. They proposed a modern view by using only one
statistic $\area$ and an appropriate sweep map $\Phi$. Then related polynomials can be constructed similarly by
defining $\dinv(D)=\area(\Phi(D))$, and $\bounce(D)= \area(\Phi^{-1}(D))$.
Several classes of polynomials constructed this way are conjectured to be jointly symmetric. See \cite[Section 6]{sweepmap}.

To attack the joint symmetry problem, we need a better understanding of the $\dinv$ or $\bounce$ statistic.
The modern view hardly helps because the construction of the sweep map is deceptively simple.
For instance, the invertibility of the sweep map for rational Dyck paths was open for over ten years,
until recently proved by Thomas-Williams in
\cite{Nathan} for a very general modular sweep map; the $\bounce$ statistic remains mysteries for rational Dyck paths.
See \cite{Rational-Invert} for further references.

Our main objective in this paper is two-folded. One is to give a geometric construction for the $\bounce$ statistic of a $\k$-Dyck path; the other is to give
a geometric interpretation of the $\dinv$ statistic of a $\k$-Dyck path. With a properly modified definition of the $\area$ statistic, we extend the
$\dinv$ sweeps to $\area$, and $\area $ to $\bounce$ result. The $\bounce$ result is inspired by a recent work of \cite{Xin-Zhang},
where a linear algorithm for $\Phi^{-1}$ was developed for $\k$-Dyck paths. Note that Thomas-Williams' general algorithm for $\Phi^{-1}$
is quadratic and is hard to be carried out by hand \cite{Nathan}. The $\dinv$ result is inspired by a recent work of \cite{dinv-area}, where
a visual interpretation of $\dinv$ was introduced.

\subsection{The sweep map and $\k$-Dyck paths}
To introduce the sweep map clearly, we use the following three models and some notions in \cite{Xin-Zhang}. For a vector $%\mathbf{k} =
 \k=(k_1,k_2,\dots,k_n)$ of $n$ positive integers, denote by $|\k|=\sum_{i=1}^nk_i$. Denote by $\CD_{\k}$ the set of all $\k$-Dyck paths. We will see that $\k$-Dyck paths reduce to ordinary Dyck paths when $k_i=1$ for all $i$, and reduce to $k$-Dyck paths when $k_i=k$ for all $i$.

Model $1$: Classical path model.
$\k$-Dyck paths are two dimensional lattice paths from $(0,0)$ to $(|\k|,|\k|)$ that never go below the main diagonal
$y=x$, with north steps of lengths $k_i$, $1\leq i \leq n$ from bottom to top, and east unit steps. Each vertex is
assigned a rank as follows. We start by assigning $0$ to $(0,0)$. This is done we add a $k_i$ as we go north with a length $k_i$ step,
and subtract a $1$ as we go east. Figure \ref{fig:model} illustrates an example of a
$\vec{k}$-Dyck path with $\k=(3,1,4,1,1)$.

\begin{figure}[!ht]
  $$
 %\hskip -1.8in
  \hskip 1.25in \oD= \hskip -1.95in\vcenter { \includegraphics[height= 2.06 in]{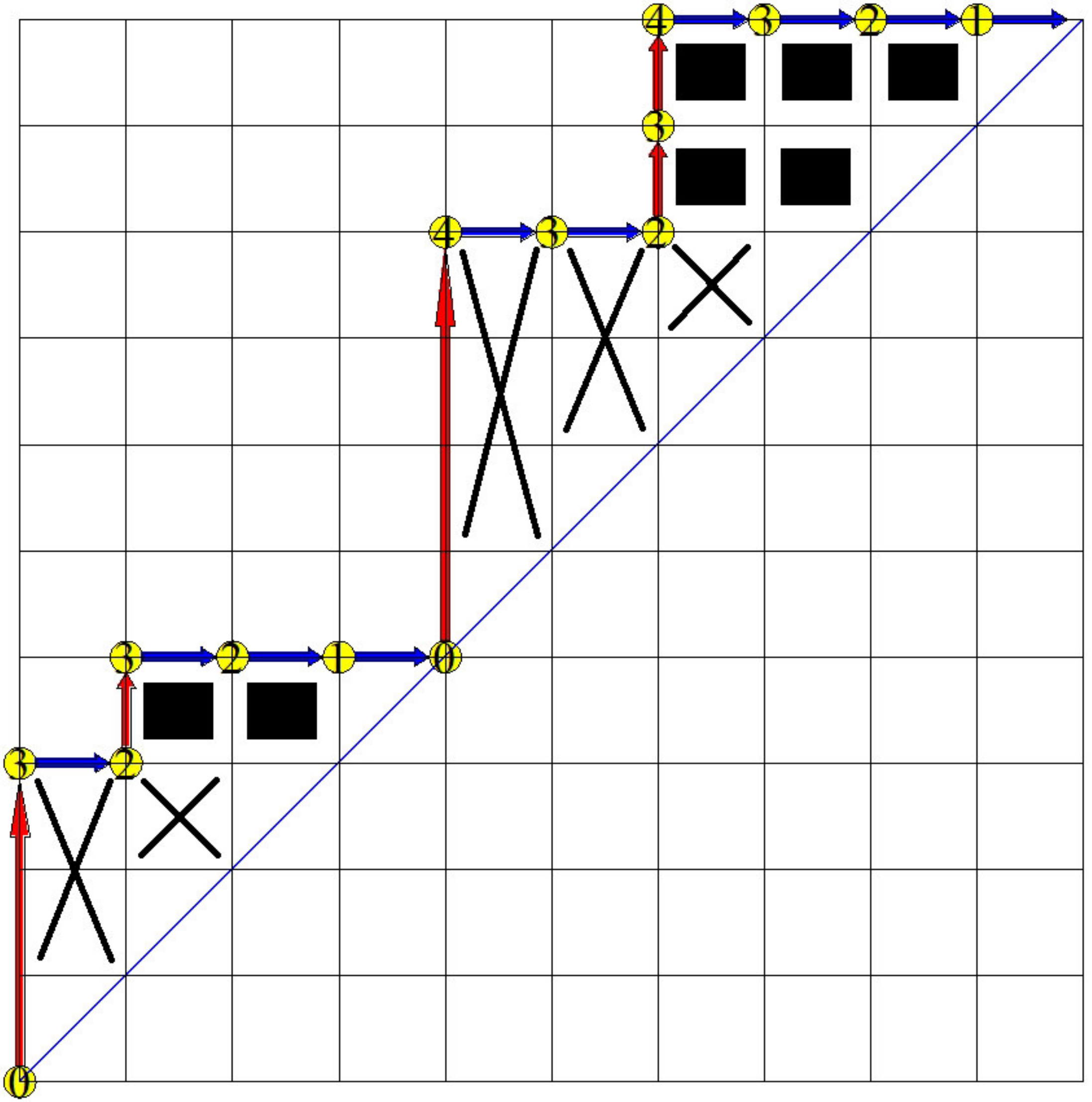}}
%\hskip -3.0in  \vcenter{ \includegraphics[height=1.4 in]{model3.png}}
$$
\caption{An example of a $\vec{k}$-Dyck path in model 1.
\label{fig:model}}
\end{figure}

Model 2: Word model. For a $\oD \in \CD_{\k}$, the SW-word $\SW(\oD)=\sigma_1\sigma_2\cdots \sigma_{|\k|+n}$ is a natural encoding of $\oD$, where $\sigma_i$ is either an $S^{k_j}$ or a $W$ depending on whether the $i$-th vertex of $\oD$ is the $j$-th South end
(of the $j$-th North step) or a West end (of an East step).  The rank is then associated to each letter of $\SW(\oD)$ by assigning  $r_1=0$ to the first letter $\sigma_1=S^{k_1}$ and for $1\le i\le |\k|+n-1$, recursively
assigning $r_{i+1}$ to be either $r_i+k_j$ if the $i$-th letter $\sigma_i=S^{k_j}$, or $r_i-1$ if otherwise $\sigma_i=W$. We can then form the two line array $\left( \SW(\oD) \atop r(\oD)\right)$.
For instance for the path $\oD$ in Figure \ref{fig:model} this gives
\begin{align*}%\label{e-II.1}
  \left(\SW(\oD) \atop r(\oD)\right)=
 \left (\begin{array}
   {ccccccccccccccc}
   S^3 \!\!\! & W \!\!\! & S^1 \!\!\! & W \!\!\! & W \!\!\! &W \!\!\! & S^4 \!\!\!&  W \!\!\! & W \!\!\! &S^1 \!\!\!& S^1 \!\!\! & W \!\!\! & W \!\!\! & W \!\!\! & W  \\
   0 \!\!\! &  3 \!\!\! &  2 \!\!\! & 3 \!\!\! &  2 \!\!\! &  1 \!\!\! &  0 \!\!\! &  4 \!\!\! &  3 \!\!\! &  2 \!\!\! &  3 \!\!\! &4 \!\!\! &  3 \!\!\! &  2 \!\!\! & 1
 \end{array} \right).
%\eqno \II.1
\end{align*}

%Now sort the columns of $(\ref{e-II.1})$ according to the second row, and let the right one comes first for equal ranks.
%Then the top row is just the SW-word $\SW(D)$ of the sweep map image of $\oD$.
%Our running example gives
%\begin{align}
%\left(\SW(D) \atop r(\overline{D})_{\texttt{sorted}}\right)= \left( \begin{array}{cccccccccccccccc}
%S \!\!\! & W \!\!\! & W \!\!\! & S \!\!\! & S \!\!\! & W \!\!\! & W \!\!\! & W \!\!\! & S \!\!\! & W \!\!\! & W \!\!\! & W \!\!\! & W \!\!\! & W \!\!\! & W \!\!\! & W \\
%0 \!\!\! & 4 \!\!\! & 8 \!\!\! & 8 \!\!\! & 8 \!\!\! & 12 \!\!\! & 12 \!\!\! & 12 \!\!\! & 12 \!\!\! & 16 \!\!\! & 16 \!\!\! & 16 \!\!\! & 20 \!\!\! & 20 \!\!\! & 20 \!\!\! & 24
%\end{array} \right).
%\end{align}

Model $3$: Visual path model.
$\k$-Dyck paths are two dimensional lattice paths from $(0,0)$ to $(|\k|+n,0)$ that never go below the horizontal axis
with up steps $($red arrows$)$ $(1,k_i)$, $1\leq i\leq n$ from left to right, and down steps $($blue arrows$)$ $(1,-1)$. It is clear that the ranks
are just the levels (or $y$-coordinates). The sweep map image $D$ of $\bD$ is obtained by reading its steps by their starting levels from bottom to top, and from right to left when at the same level.
This corresponds to sweeping the starting points of the steps from
bottom to top using a line of slope $\epsilon$ for sufficiently small $\epsilon >0$.
The visualization of the ranks in this model allows us to have better understanding of many results.
For instance for the path $\oD$ in Figure \ref{fig:model} this gives
\begin{figure}[!ht]
  $$
 %\hskip -1.8in
  %\hskip 1.25in \oD= \hskip -1.95in\vcenter { \includegraphics[height= 2.06 in]{model1.png}}
\hskip .4 in
  \oD = \hskip -.5 in \vcenter{ \includegraphics[height=1.4 in]{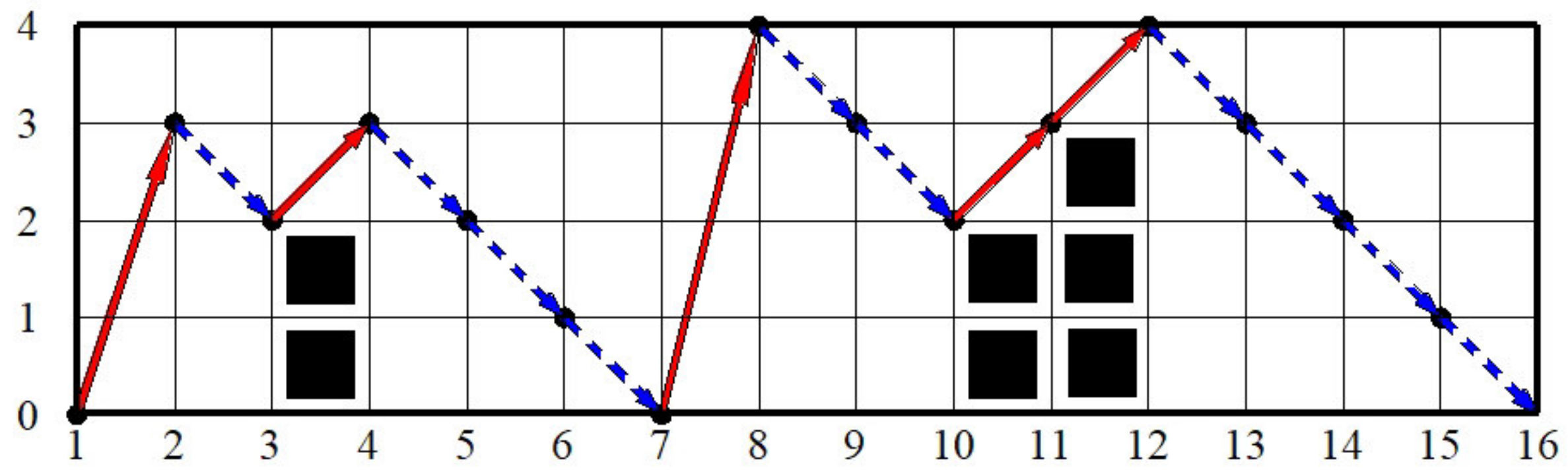}}
$$
\caption{An example of a $\vec{k}$-Dyck path in model 3. The sweep $\dinv$ is 13 and the red $\dinv$ is 3. Thus we have
$\dinv(\oD)=16.$
\label{fig:mode3}}
\end{figure}

The sweep map of a $\k$-Dyck path is usually a $\k'$-Dyck path where $\k'$ is obtained from $\k$ by permuting its entries. Denote by $\mathcal{K}$ the set of
all such $\k'$ and by $\CD_{\mathcal{K}}$ the union of $\CD_{\k'}$ for all such $\k'$. The sweep map is a bijection from $\CD_{\mathcal{K}}$ to itself.

A Dyck path $\bD\in D_{\k}$ may be encoded as $\bD=(a_1,a_2,\dots,a_{|\k|+n})$ with each entry either
$k_i$ or $-1$. The $SW$-word of the $\bD$ is $\SW(\bD)=\sigma_1\sigma_2\cdots\sigma_{|\k|+n}$ where $\sigma_j=S^{k_i}$ if $a_j=k_i$
and $\sigma_j=W$ if $a_j=-1$. The rank sequence $r(\bD)=(0=r_1,r_2,\dots,r_{|\k|+n})$ of $\bD$ is defined as the partial sums
$r_i=a_1+a_2+\cdots+a_{i-1}\geq 0$, called starting rank (or level) of the $i$-th step.
Geometrically, $r_i$ is just the level or $y$-coordinate of the starting point of the $i$-th step. We also need to consider the end rank sequence
$\r(\bD)=(\r_1,\r_2,\dots,\r_{|\k|+n})=(r_2,r_3,\dots,r_{|\k|+n},0)$.
When clear from the context, we usually write $S^{k_i}$ as $S$, and denote by $r(S)$ and $\r(S)$ its starting rank and end rank, respectively.
The length of $S$ is written as $\ell(S)=\r(S)-r(S)$.

We will frequently use two orders on the arrows $A$ and $B$ of a Dyck path $D$: i) $A<B$ (under the natural order) means that
 $A$ is to the left of $B$ in $D$; ii)
$A <^s B$ (under the sweep order) means that $r(A)<r(B)$ or $r(A)=r(B)$ and $B < A$.

%\begin{theo}\label{dinv-area-bounce}
%The sweep map takes dinv to area and area to bounce for $\k$ Dyck paths. That is, for any $\overline{D}\in D_{\k}$ with $D=\Phi(\overline{D})$,
%we have $\dinv(\overline{D})=\area(D)$ and $\area(\overline(D))=\bounce(D)$.
%\end{theo}

The paper is organized as follows. In this introduction, we have introduced the basic concepts.
In Section \ref{s-define-main} we define the three statistics $\area$, $\dinv$ and $\bounce$ for $\k$-Dyck paths and state our main
result in Theorem \ref{dinv-area-bounce1}. The proof of the theorem is given in the next two sections:
Section \ref{s-proof-dinvtoarea} proves that $\dinv$ sweeps to $\area$, and
Section \ref{s-proof-areatobounce} proves that $\area$ sweeps to $\bounce$. Finally Section \ref{s-summary} gives a summary and a conjecture on $q,t$ symmetry.

\section{Area, Dinv, and Bounce for $\k$-Dyck paths\label{s-define-main}}
Throughout this section, $\k=(k_1,k_2,\dots,k_n)$ is a fix vector  of $n$ positive integers, unless specified otherwise.
We define the three statistics for $\k$-Dyck paths.
The $\area$ and $\bounce$ are defined using model 1, and the $\area$ and $\dinv$ are defined using model 3. The two $\area$ definition are easily seen to be equivalent.
We also consider the $q,t$ symmetry property.

\subsection{The Area statistic for $\k$-Dyck paths}
Traditionally, the $\area$ of a rational Dyck path is defined to be the number of complete lattice cells between the path and the main diagonal.

We define the $\area$ statistic of a
$\vec{k}$-Dyck path $D$ to be equal to the sum of the starting ranks of all north steps of $D$.
In model 1, this is the number of complete lattice cells between the path and the main diagonal, and in rows containing a south end of a north step;
In model 3, this is the number of complete lattice cells between the red arrows and the horizontal axis. For example, in Figures \ref{fig:model} and \ref{fig:mode3}, we have $\area(\oD)=7$.
Note that some of the complete lattice cells with crosses are not counted, because their rows do not contain a south end of a north step.

This definition is closely related to the $\dinv$ statistic in the next subsection.
It agrees with the $\area$ for ordinary Dyck paths.

\subsection{The dinv statistic for $\k$-Dyck paths}
Our $\dinv$ sweeps to $\area$ result is inspired by \cite[Proposition 4]{dinv-area} for $(m,n)$-Dyck paths. In that paper,
the authors gave a geometric description of the dinv statistic and a representation of the area by ranks.
Our $\area$ definition mimics that area formula. We follow some notations there.

By abuse of notation, we will use $W_i$ (resp. $S_j$) for the $i$-th blue (resp. $j$-th red) arrow for a $\k$-Dyck path $D$. Then we have
$$ \area (D) = \sum_{S_j} r(S_j),$$
where the sum ranges over all red arrows $S_j$ of $D$. Compare this formula with \cite[Theorem 2]{dinv-area} for $(m,n)$-Dyck paths.

The $\dinv$ statistic of a Dyck path $D\in D_{\k}$ needs a correction term which we call the red $\dinv$. More precisely, the $\dinv$ consists of two parts that can be described geometrically as follows.

\begin{enumerate}
  \item Sweep $\dinv$: Each pair $(W_i,S_j)$ with $W_i<S_j$ contributes $1$ if $W_i$ sweeps $S_j$, denoted $W_i\rightarrow S_j$, which means
$W_i$ intersects $S_j$ when we move it along a line of slope $\epsilon$
       (with $0 < \epsilon <\!\!<1$) to the right past $S_j$;
  \item Red $\dinv$: Each  pair $(S_i,S_j)$ of red arrows with $S_i < S_j$ contributes $\r(S_j)-\r(S_i)$ if
        $r(S_i)\geq r(S_j)$ and $\r(S_j) > \r(S_i)$, and contributes $\r(S_i)-\r(S_j)$ if $r(S_i) < r(S_j)$ and $\r(S_j) < \r(S_i)$.
In other words, each pair $(S_i,S_j)$ of red arrows contributes $|\r(S_j)-\r(S_i)|$ if one of the two arrows can be contained in the other
by moving them along a line of slope $\epsilon$.
\end{enumerate}

In formula we have
\begin{align*}
\dinv(D)= \sum_{W_i<S_j} \chi(W_i\rightarrow S_j) &+ \sum_{S_i<S_j} \chi(r(S_i)\ge r(S_j) \;\&\; \r(S_j)>\r(S_i)) (\r(S_j)-\r(S_i))     \\
  &+\sum_{S_i<S_j} \chi(r(S_i)< r(S_j) \;\&\; \r(S_j)<\r(S_i)) (\r(S_i)-\r(S_j)).    %\label{e-dinv-x}
\end{align*}

For example, in the Figure $\ref{fig:mode3}$, we have $\dinv(\oD)=16$.

\subsection{The bounce statistic for $\k$-Dyck paths}
The $\bounce$ statistic was defined by Haglund for ordinary Dyck paths and extended by Loehr for $k$-Dyck paths. We will extend Leohr's bounce path to that of $\k$-Dyck paths
with the help of an intermediate rank tableau $R^b(D)$, which will be proved to be the rank tableau $R(D)$ in \cite{Xin-Zhang}. The bounce paths for rational Dyck paths
are still unknown.

The bounce path is a sequence of alternating \emph{vertical~moves} and \emph{horizontal~moves} constructed with the help of
an intermediate rank tableau $R^b(D)$ consisting of $n$ columns with $k_i+1$ cells in the $i$-th column.
The entries in each column will be of the form $a,a+1,a+2,\dots$ from top to bottom, so to construct $R^b(D)$ it suffices to determine the top row entries.

We begin at $(0,0)$ with a vertical move, and eventually end at $(|\vec{k}|, |\vec{k}|)$ after a horizontal move.
Let $v_0, v_1,\cdots$ denote the number of passing north steps of the successive vertical moves and let $h_0, h_1, ...$ denote the number of passing east steps of the successive horizontal moves.
These numbers are calculated in the following algorithm.

%\begin{algoo}[Bouncing Algorithm]
%\label{al-Bouncing Algorithm}

\noindent
\textbf{Bouncing Algorithm}

\medskip
\noindent
Input: A $\k$-Dyck path $D\in D_{\vec{k}}$ in model $1$.

\noindent
Output: The bounce path of $D$, $\bounce(D)$, and the rank tableau $R^b(D)$.

\begin{enumerate}
\item  To find $v_0$, move due north from $P_0=(0,0)$ until you reach the west end $Q_0$ of an east step of the Dyck path $D$; the number of north steps traveled is $v_0$.
Write $v_0$ zeroes in turn in the first row in $R^0$, add one in the lower cells from top to bottom in each column to obtain $R^1$.
Let $h_0$ be the number of $1$'s in $R^1$ and move due east $h_0$ units to a position $P_1$.

%
%\item Next, move north from the current position $P_1$ until you reach the west end $Q_1$ of an east step of the Dyck path; let $v_1$ be the number of north steps traveled.
%Write $v_1$ ones in turn in the first row in $R^1$, add one in the lower cell from top to bottom in each column to obtain $R^2$.
%Let $h_1$ be the number of $2$ in $R^2$ and move due east $h_1$ units to a position $P_2$.

\item Suppose in general we reached a position $P_i$ and need to find $v_i$.
Then we move north from $P_i$ until we reach the west end $Q_i$ of an east step of the
Dyck path. Define $v_i$ to be the number of north steps traveled.
Write $v_i$ (possibly equal to $0$) $i$'s in turn in the first row in $R^i$, add one in the lower cell from top to bottom in each column to obtain $R^{i+1}$.
Let $h_i$ be the number of $i+1$ in $R^{i+1}$ and move east $h_i$ units to a position $P_{i+1}$.

\item Proceed as above until we eventually end at $P_f=(|\vec{k}|, |\vec{k}|)$. The final tableau $R^f$ is our rank tableau $R^b(D)$,
and the $\bounce$ statistic is defined
to be $$\bounce(D)=\sum_{i\geq0}i\times v_i(D)$$
a weighted sum of the lengths of the vertical moves in the bounce path derived from $D$.
Equivalently, $\bounce(D)$ is also the sum of the entries in the first row of $R^b(D)$.
\end{enumerate}
%\end{algoo}

\medskip
We illustrate the bounce statistic by the following Figure \ref{fig:bouncepath}, where the Dyck path $D$ is the sweep map image of
the path $\oD$ in Figure \ref{fig:model}. To obtain the bounce path $(\bounce(D))$ and the rank tableau $R^b(D)$, We first find $v_0=2$. Then we
construct the tableau $R^1$ with two columns. Thus $h_0=2$ and we reach the position $P_1=(2,7)$, as shown in the Figure.
Now we are blocked by the path, so $Q_1=P_1$, which means $v_1=0$, and hence $R^2 =R^1$.
It follows that $h_1=2$ and we reach the position $P_2=(4,7)$, as shown in the Figure. Continuing this way, it is easy to obtain
$R^3$, $R^4$, and the bounce path $(\bounce(D))$. The rank tableau $R^b(D) = R^4$.

\begin{figure}[!ht]
  $$
 \hskip -1.8in
 \vcenter{ \includegraphics[height= 2.06 in]{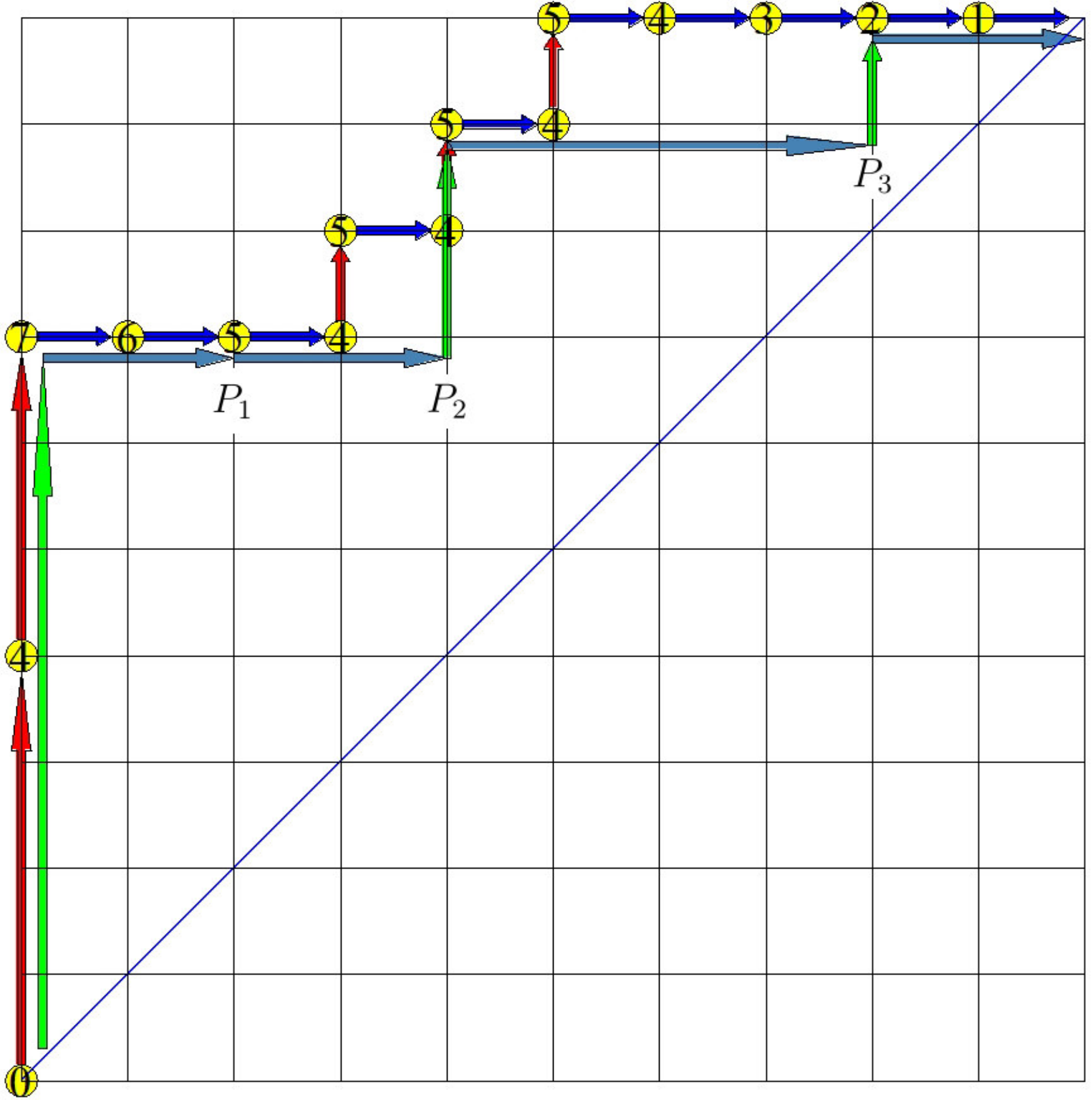}}
\hskip -3.0in  \vcenter{ \includegraphics[height=1.8 in]{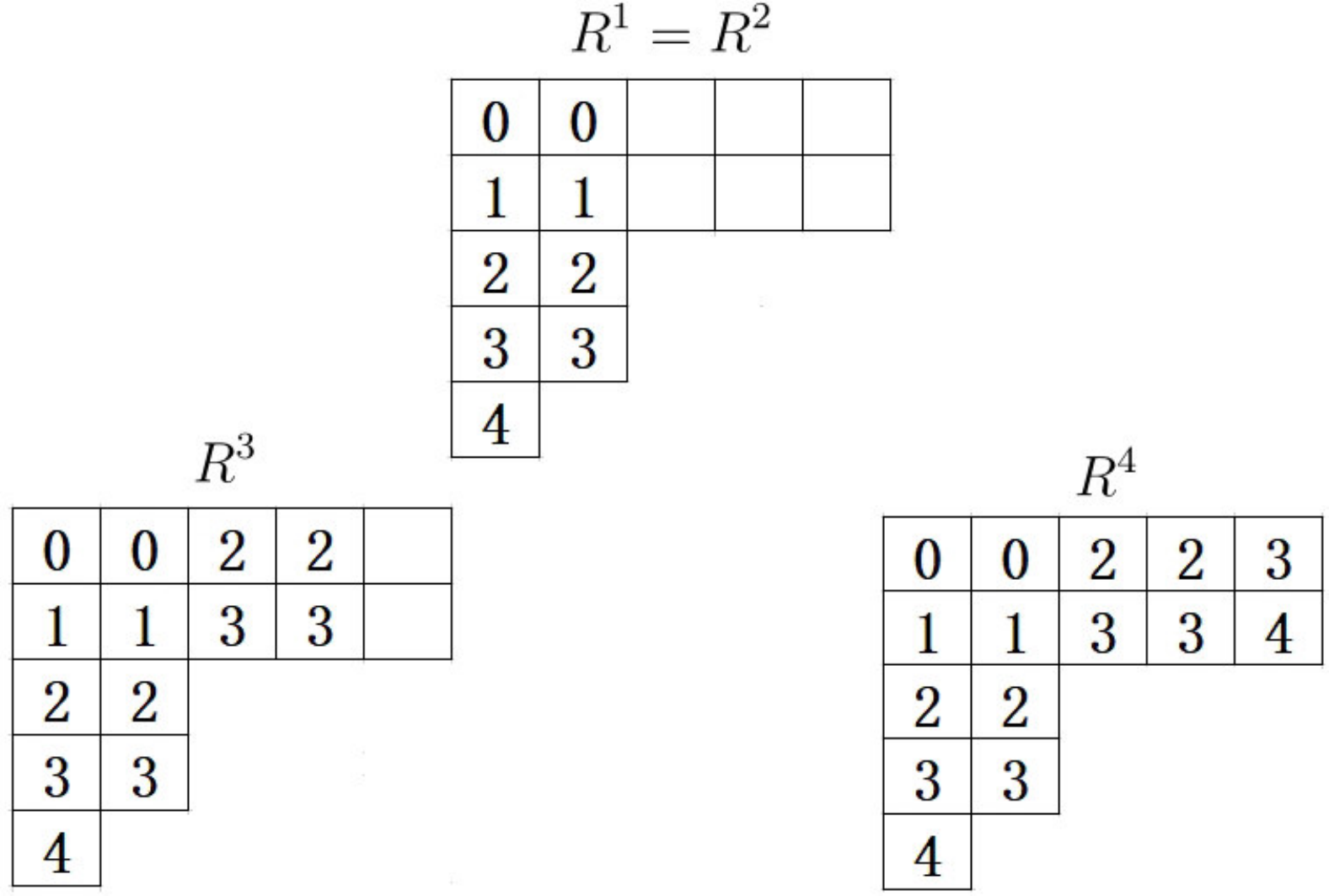}}
$$
\caption{An example of the bounce path and rank tableau for the Dyck path $D=\Phi(\oD)$, where $\oD$ is depicted in Figure \ref{fig:model}.
\label{fig:bouncepath}}
\end{figure}

Now we need to show that the bounce path is always well-defined.

Note that, for a Dyck path $D\in D_{\vec{k}}$, the bounce path does not necessarily return to the diagonal $x=y$ after
each horizontal move.
Consequently, it may occur that $P_i$ is the starting point of an east step of $D$, so $v_i=0$. We claim that
$h_i>0$ in this case for $i<f$. Then we move forward to $P_f$ without a stop.
Assume to the contrary that $h_i=0$. Then there are no $i+1$ in $R^i$, and hence no larger ranks also.
This means that we have moved $\sum_{j=1}^{v_0+v_1+\cdots+v_{i-1}}k_j$ east steps in total. Since the lengths of the north steps
is $$%\sum_{j=1}^{i-1}H_j=
\sum_{j=1}^{v_0}k_j+\sum_{j=v_0+1}^{v_0+v_1}k_j+\cdots+\sum_{j=v_0+\dots+v_{i-2}+1}^{v_0+\dots+v_{i-1}}k_j=
\sum_{j=1}^{v_0+\dots+v_{i-1}}k_j,$$
$P_i$ is on the diagonal line. But then the east step starting at $P_i$ will go below the diagonal line.
This contradicts the fact that $D$ is a Dyck path.

Our bounce path reduces to Loehr's bounce path for $k$-Dyck paths.

\begin{theo}\label{dinv-area-bounce1}
The sweep map takes $\dinv$ to $\area$ and $\area$ to $\bounce$ for $\k$-Dyck paths.
That is, for any Dyck path $\overline{D}\in  \CD_{\mathcal{K}}$ with sweep map image $D=\Phi(\overline{D})$,
we have $\dinv(\overline{D})=\area(D)$ and $\area(\oD)=\bounce(D)$.
\end{theo}

\subsection{About the $q,t$-symmetry}
A vector $\k=(k_1,\dots, k_n)$ of positive integers is also called an ordered partition. Arranging its entries decreasingly gives a partition, called the partition
$\lambda(\k)$ of $\k$. %For a partition $\lambda$, define $\mathcal{K}(\lambda)=\{ \k: \lambda(\k)=\lambda\}$ to be the set of all vectors $\k$ whose partition is $\lambda$}.
We can define $q,t$-Catalan numbers of type $\lambda$ similarly as follows:
\begin{align*}
  C_{\lambda}(q,t): = \sum_{\lambda(\k)=\lambda}\sum_{ D \in \CD_{\k}} q^{\dinv(D)} t^{\area(D)}=\sum_{\lambda(\k)=\lambda}\sum_{ D \in \CD_{\k}} q^{\area(D)} t^{\bounce(D)},
\end{align*}
where the sum ranges over all $\k$-Dyck paths satisfying $\lambda(\k)=\lambda$. Clearly, $C_{k^n}(q,t)$ agrees with Loehr's higher $q,t$-Catalan polynomials, where $k^n$ denotes
the partition consisting of $n$ equal parts $k$.

We investigate the $q,t$ symmetry of $C_\lambda(q,t)$, and report as follows.

We do have the $q,t$ symmetry for partitions $\lambda$ of length $n=2$.
We can prove this property easily as follows.
For $\k=(k_1,k_2)$, Dyck paths $D$ are uniquely determined by the two ranks
$(r_1=0,r_2)$ of the two red arrows. Let us call them the red ranks.
The path $D$ starts with a red arrow $S^{k_1}$ followed by $k_1-r_2$ blue arrows $W$, then a red arrow $S^{k_2}$ followed
by $k_2+r_2$ blue arrows $W$.

It is easily checked that $\Phi^{-1}(D)$ has red ranks $(r_1=0, k_1-r_2)$ for $0\leq r_2 \leq k_1$,
but when $r_2=k_1$, $\Phi^{-1}(D)$ starts with $S^{k_2}$ instead of $S^{k_1}$. %µ«²¢²»Ó°ÏìbounceµÄÖµ.
%{\color{red}It follows that the contribution of $D$ in $C_{\lambda}(q,t)$ is $q^{r_2} t^{k_1-r_2}$ when $r_2\ge 1$, and $q^0t^{k_2}$ when $r_2=0$, by using the bounce formula.
%Thus we can define the map
%$\pi: D\mapsto \pi(D) \in \CD_{\k}$, where $\pi(D)$ is determined by its two red ranks $(0,k_1-r_2)$ when $r_2\ge 1$, and $(0,k_2)$ when $r_2=0$ (switching $S^{k_1}$ and $S^{k_2}$ additionally). The map $\pi$ shows the $q,t$ symmetry of
%$C_{\mathcal{K}}(q,t)$}.
It follows that the contribution of $D$ in $C_{\lambda}(q,t)$ is $q^{r_2} t^{k_1-r_2}$
by using the bounce formula.
Thus we can define the map
$\pi: D\mapsto \pi(D) \in \CD_{\k}$, where $\pi(D)$ is determined by its two red ranks $(0,k_1-r_2)$.
The map $\pi$ shows the $q,t$ symmetry of
$C_{\lambda}(q,t)$.

\medskip
For partitions $\lambda$ of length $n=3$,
computer experiment suggests that $C_{\lambda}(q,t)$ is symmetric in $q,t$. We obtain
explicit bounce formula as follows. The dinv formula does not seem nice.

For $\k=(k_1,k_2,k_3)$, Dyck paths $D\in \CD_{\k}$ are uniquely determined by their red ranks $(r_1=0,r_2,r_3)$. We have
$$\bounce(D)=\left\{
               \begin{array}{ll}
                 2(k_1-r_2)+r_2+k_2-r_3-\min(r_2,k_2), & \text{ if } r_2+k_2-r_3 \geq 2\min(r_2,k_2) ; \\
                 2(k_1-r_2)+ \lceil {\frac{r_2+k_2-r_3} 2}\rceil, & \text{otherwise.}
               \end{array}
             \right.
$$
We have verified the $q,t$ symmetry for almost all cases for which $C_{\lambda}(1,1)<3\times 10^4$.

A combinatorial proof seems out of reach at this moment. We believe that it is very hopeful to prove the $q,t$ symmetry property in this case by MacMahon's partition analysis technique.

For partition $\lambda$ of length $n\ge 4$, the $q,t$ symmetry no longer holds.
The smallest case that violates the $q,t$ symmetry is when $\lambda=(3,1,1,1)$. In this case, we have
\begin{multline*}
  C_{\lambda}(q,t)-C_{\lambda}(t,q)={q}^{6}{t}^{3}-{q}^{3}{t}^{6}-{q}^{6}{t}^{2}-2\,{q}^{5}{t}^{3}\\
+2\,{q}^
{3}{t}^{5}+{q}^{2}{t}^{6}+2\,{q}^{5}{t}^{2}+{q}^{4}{t}^{3}-{q}^{3}{t}^
{4}-2\,{q}^{2}{t}^{5}-{q}^{4}{t}^{2}+{q}^{2}{t}^{4}.
\end{multline*}

Another example is when $\lambda=(3,2,2,1)$. We have
\begin{multline*}
  C_{\lambda}(q,t)-C_{\lambda}(t,q)
={q}^{9}{t}^{3}-{q}^{3}{t}^{9}-{q}^{9}{t}^{2}-{q}^{8}{t}^{3}+{q}^{6}{t}^{5}-{q}^{5}{t}^{6}+{q}^{3}{t}^{8}+{q}^{2}{t}^{9}+
{q}^{8}{t}^{2}-{q}^{7}{t}^{3}\\
-{q}^{6}{t}^{4}+{q}^{4}{t}^{6}+{q}^{3}{t}^{7}-{q}^{2}{t}^{8}+
{q}^{7}{t}^{2}+{q}^{6}{t}^{3}+{q}^{5}{t}^{4}-{q}^{4}{t}^{5}-{q}^{3}{t}
^{6}-{q}^{2}{t}^{7}-{q}^{6}{t}^{2}+{q}^{2}{t}^{6}.
\end{multline*}
%
%
%For $\k=[1,1,2,3]$, %the $\mathcal{K}$ and $\CD_{\mathcal{K}}$ will be defined the following.
%the polynomial is not symmetry, then
%
%$C_{\mathcal{K}}(q,t)-C_{\mathcal{K}}(t,q)$
%$={q}^{9}{t}^{3}-{q}^{3}{t}^{9}-{q}^{9}{t}^{2}-{q}^{8}{t}^{3}+{q}^{6}{t}
%^{5}-{q}^{5}{t}^{6}+{q}^{3}{t}^{8}+{q}^{2}{t}^{9}+{q}^{8}{t}^{2}-{q}^{
%7}{t}^{3}-{q}^{6}{t}^{4}+{q}^{4}{t}^{6}+{q}^{3}{t}^{7}-{q}^{2}{t}^{8}+
%{q}^{7}{t}^{2}+2\,{q}^{6}{t}^{3}+{q}^{5}{t}^{4}-{q}^{4}{t}^{5}-2\,{q}^
%{3}{t}^{6}-{q}^{2}{t}^{7}-2\,{q}^{6}{t}^{2}-2\,{q}^{5}{t}^{3}+2\,{q}^{
%3}{t}^{5}+2\,{q}^{2}{t}^{6}+2\,{q}^{5}{t}^{2}+{q}^{4}{t}^{3}-{q}^{3}{t
%}^{4}-2\,{q}^{2}{t}^{5}-{q}^{4}{t}^{2}+{q}^{2}{t}^{4}
%$
For most partitions $\lambda$ of length $n\ge 4$, $C_\lambda(q,t)$ are not $q,t$ symmetric, but we do have a conjecture
stated in Section \ref{s-summary}.

\section{proof that dinv sweeps to area\label{s-proof-dinvtoarea}}

%\subsection{Some basic auxiliary facts}
%We need some concepts and a number of auxiliary properties of the sweep map for $\k$-Dyck paths.

\begin{prop} [in \cite{dinv-area}] \label{prop-area}
The starting rank of any arrow $A$ of $D$ may be simply obtained by drawing a line of slope $0<\epsilon\ll 1$ at the starting point of its preimage $\bar{A}$, then counting the lengths of the red arrows starting below the line and minus the number of the blue arrows that start below the line. In formula, we have
\begin{align*}
  r(A) = \sum_{\bar{S}<^s \bar{A}} \ell(\bar{S}) - \# \{ \bar{W} : \bar{W}<^s \bar{A}\}.
\end{align*}
\end{prop}

\begin{prop}[Zero-row-count property in \cite{dinv-area}]\label{prop--row-zero}
In model $3$, each lattice cell may contain a segment of a red arrow or a segment of a blue arrow or no segment at all. The red segment count of row $j$ will be denoted $c^r(j)$ and the blue segment count is denoted $c^b(j)$. We will denote by $c(j)=c^r(j)-c^b(j)$ and refer to it the $j$-th row count. Observe that in every row of a path diagram , the red segments and blue segments have to alternate. Every row must start with a red segment and end with a blue segment, and hence $c(j)=0$ holds for all $j$.
\end{prop}

%\begin{theo}\label{area-sum}
%For any $\mathbf{k}$-Dyck path $D$, we have
%$$
%area(D)=\sum^{|\mathbf{k}|}_{j=1}r(S_j(D)),
%$$
%where "$r(S_j(D))$" denotes the rank of the $j^th$ South end of $D$.
%\end{theo}

%\begin{figure}[!ht]
%  \begin{center}
% \includegraphics[height= 2.0 in]{HatS-W.png}
%  \end{center}
%  \caption{$r(\S)=i$.
%  \label{fig:hatS-W}}
%\end{figure}
%%%%Here I want to try the general way: just choose the rightmost red arrow with the largest starting rank. I think it will not be too complicated.

%\subsection{proof that dinv sweeps to area}
\begin{proof}[Proof of Theorem \ref{dinv-area-bounce1} part 1]
 We follow the idea in \cite{dinv-area}. We will prove $\dinv(\bD)=\area(D)$ by induction on $\area(\bD)$.

%The base case where $dinv(\overline{D})=0$. The $\vec{k}$-Dyck path with dinv $0$ is the path $\overline{D}$ that
%red arrows $S^{k_{1}},S^{k_{2}},...,S^{k_{n}}$  and then  $|\k|$ blue arrow $W$.
%This forces the image $D$ of $\overline{D}$ to a red arrow $S^{k_1}$ followed by $k_1$ blue arrow $W$, then a red arrow $S^{k_2}$ followed by $k_2$ blue arrow $W$, and so on. Thus the starting ranks of all north steps are zero.

The base case is when $\area(\bD)=0$. Such a $\k$-Dyck path $\bD$ is uniquely given by:
a red arrow $S^{k_1}$ followed by $k_1$ blue arrows $W$, then a red arrow $S^{k_2}$ followed by $k_2$ blue arrows $W$, and so on.
The sweep map image $D$ of $\overline{D}$ is clearly given by: red arrows $S^{k_{n}},S^{k_{n-1}},...,S^{k_{1}}$ followed by $|\k|$ blue arrows $W$.

To show $\dinv(\bD) = \area(D)$ in this case, we compute as follows. The $\area$ of $D$ is simply given by
\begin{align*}
\area(D) = (n-1)k_n+(n-2)k_{n-1}+\cdots+k_2.
\end{align*}

The dinv formula in this case simplifies as follows.
\begin{align*}
\dinv(\bD)&=\sum_{j\ge 2}    \sum_{W_i<S_j} \chi(W_i\rightarrow S_j) + \sum_{S_i < S_j} \chi( k_j > k_i)(k_j - k_i). %\label{dinv-base}
\end{align*}
For each $S_j=S^{k_j}$,
we group the arrow $S^{k_t}$ together with the followed $k_t$ blue arrows $W$, and compute their contribution in the above dinv formula.
For each $t<j$, we have two cases:

Case 1 when $k_t \geq k_j$: only the final $k_j$ blue arrows sweep $S^{k_j}$, contributing $k_j$ sweep dinvs; the red $\dinv$ is clearly $0$.
Thus the total contribution in this case is $k_j$.

Case 2 when $k_t < k_j$: the $k_t$ blue arrows sweep $S^{k_j}$, contributing $k_t$ sweep dinvs; the red $\dinv$ is clearly $k_j-k_t$. Thus the total contribution
in this case is still $k_j$.

It follows that
$$\dinv(\bD)= (n-1)k_n+(n-2)k_{n-1}+\cdots+k_2 = \area(D).$$

Now assume $\area(\bD)>0$. Then we choose the rightmost red arrow $\S$ with the largest rank among the red arrows.
Then $\S$ must be followed by a blue arrow, denoted $\W$. By switching the two arrows $\S$ and $\W$ in $\bD$
(denoting them by $\W'\S'$), we subtract one cell from $\bD$ and obtain another $\k$-Dyck path $\bD'$.
Clearly $\area(\bD')$ is one less than $\area(\bD)$. Let $D'$ be the sweep map image of $\bD'$. Then by the induction hypothesis,
$\dinv(\bD')=\area(D')$. We will show in Lemma \ref{lem-induction} that
$\dinv(\bD)-\dinv(\bD')=\area(D)-\area(D')$. The theorem is then proved.
\end{proof}

\begin{lem}\label{lem-induction}
We subtract one cell by switching the two arrows $\S$ and $\W$ in $\bD$ and obtain another $\k$-Dyck path $\bD'$. Let $D'$ be the sweep map image of $\bD'$.
We have the equation
$$\dinv(\bD)-\dinv(\bD')=\area(D)-\area(D').$$
\end{lem}

To prove Lemma \ref{lem-induction}, it suffices to prove the following Propositions \ref{rec-area} and \ref{rec-dinv}.
The former computes the difference $\area(D)-\area(D')$ and the latter computes the difference $\dinv(\bD)-\dinv(\bD')$.

Let $\bD$ and $\bD'$ be as above. In what follows in this section, we shall also suppose $r(\S)=i$ and $\ell(\S)=k$
and use $\S,\W,\S', \W'$ described below unless specified otherwise.
We have depicted in the left picture of Figure \ref{fig:the-cell} the cell that we have subtracted from the preimage $\overline{D}$
to obtain $\overline{D}'$. Replacing the $\S\W$ in $\oD$ by the dashed arrows $\W'\S'$ gives $\oD'$. Clearly $\ell(\S)=\ell(\d)=k$.
Let $D$ and $D'$ be the sweep map images of $\overline{D}$ and $\overline{D}'$.
%We have depicted the red arrow $\S$ in Figure $\ref{fig:HatS-W}$.
%Then we will give some the facts. %about dinv statistic.

%\begin{figure}[!ht]
%  \begin{center}
% \includegraphics[height= 1.4 in]{areaneed.png}
%  \end{center}
%  \caption{Contribution for the area difference.
%  \label{fig:the-cell}}
%\end{figure}

\begin{figure}[!ht]
  $$
 \hskip -1.5in
 \vcenter{ \includegraphics[height= 2.36 in]{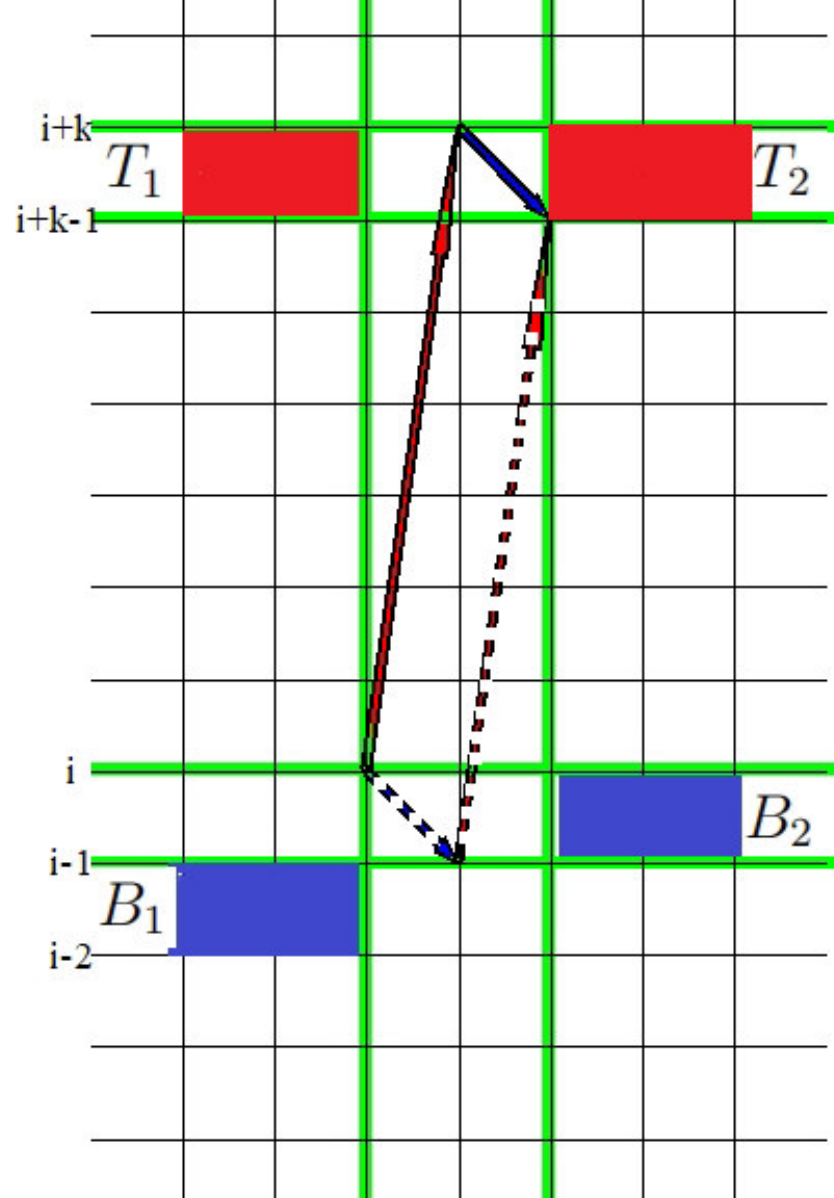}}
\hskip -3.3in  \vcenter{ \includegraphics[height=2.36 in]{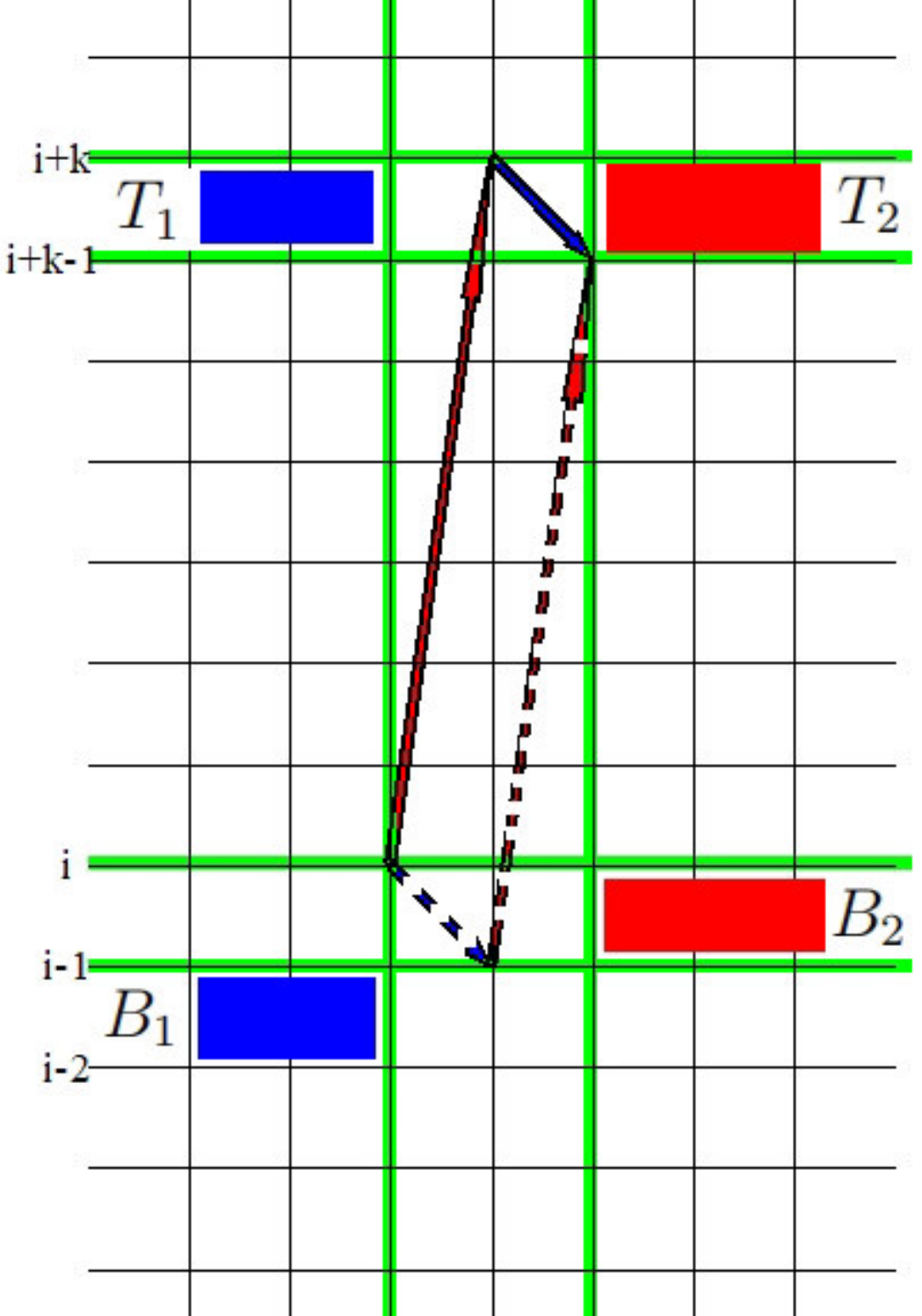}}
$$
\caption{ Contribution for the area difference and dinv difference.
\label{fig:the-cell}}
\end{figure}

It is convenient to use the following notations.
For any set $\U \subseteq \mathbb{N}$ and $a \in \mathbb{N}$,
we denote by $S^\L_{\U}=\{S|S< \S, r(S)\in \U \}=\{S|S< \S', r(S)\in \U \}$, where the first set is in $\bD$ and the second set is in $\bD'$.
The equality is clear and we will not distinguish whether the set is for $\bD$ or $\bD'$.

We list some similar notations as follows.
$$S^\L_{\U}=\{S|S< \S, r(S)\in \U \}   ,\quad S_{\U}^{\L, a}=\{S|S\in S^\L_{\U}, \r(S)\geq a\},$$
$$S^\R_{\U}=\{S|\S < S, r(S)\in \U\}, \quad  S^{\R,a}_{\U}=\{S|S\in S^\R_{\U}, \r(S)\geq a\},$$

$W^\L_{\U} = \{W|W < \S, r(W) \in \U \}$, $W^\R_{\U} = \{W|\S < W, r(W) \in \U \}$.
The notations $S_\U^\L,S_\U^{\L,a},S_\U^\R$ and $S_\U^{\R,a}$ are same with $S< \d$ or $\d < S$ in $\bD'$.
%$W^\L_{\U}=\{W|W< \S, r(W)\in \U \}$ and $W^\R_{\U}=\{W|\S < W, r(W)\in \U\}$.

Then the following two properties about set  $S_\U^\L,S_\U^{\L,a}$.\\
Let $\U = \U_1 \bigcup \U_2$,   $\U_1\bigcap \U_2 = \emptyset $, $\U_1, \U_2 \subseteq \mathbb{N}$.

$1.$ $S_\U^\L = S_{\U_1}^\L \bigcup S_{\U_2}^\L.$

$2.$ $|S_\U^\L/S_\U^{\L,a}|=|S_\U^\L|- |S_\U^{\L,a}|.$

Now our task is to determine the difference
$$\area(D)-\area(D') = \sum_{ S\in \bD} r(S)  -\sum_{S\in \bD'} r(S)
$$
using $\bD$ and $\bD'$ by means of Proposition \ref{prop-area}.

\begin{prop}\label{rec-area}
Let $\overline{D'}$ be obtained from $\overline{D}$ by removing one area cell as above, and let $D$ and $D'$ be their sweep map images. Then
\begin{align}
\area(D)-\area(D')& = \sum_{S\in S^\L_{\{i-1\}}}\ell(S) - k \times |S_{\{i-1\}}^\L|+|S^\L_{\{i\}}|  - c^b(B_1)-c^b(B_2). \nonumber
    %&=-\sum_{ S \in S_{\{i-1\}}^\L}(k-\ell(S))+|S^\L_{\{i\}}| - c^b(B_1)-c^b(B_2).
\end{align}
\end{prop}

%our task is to determine the difference
%$\area(D)-\area(D') = \sum_{\S\neq S\in \bD}$ using $\bD$ and $\bD'$ by means of Lemma .
\begin{proof}
For each red arrow $S\neq \S$ in $\bD$, it is also in $\bD'$. We need to compute the difference of its corresponding ranks in $D$ and $D'$. Clearly, this difference is given by
$$ k \chi(\S \m S) - k \chi(\S'\m S) - \chi(\W \m S)+ \chi(\W' \m S), $$
since all the other terms cancel.

This can be simplified as
$$ -k \chi(\S'\m S \m \S) + \chi( \W' \m S \m \W),$$
since $\S' \m \S$ and $\W' \m \W$. Now $\chi(\S'\m S \m\S)=1$ only when $S\in S_{\{i-1\}}^\L$ and $\chi( \W' <^s S < \W)=1$ only when $S \in S^\L_{\{i\}}$.

By summing over all such $S$, the difference becomes
%$$-k \times (|S_{\{i-1\}}^\L|+ |S_{\{i\}}^\R|)+|S^\L_{\{i,i+1,\dots,i+k-1\}}|+ |S^\R_{\{i+1,i+2,\dots,i+k\}}|.$$
\begin{align}
-k \times |S_{\{i-1\}}^\L|+|S^\L_{\{i\}}|.\label{area-diff-noS}
\end{align}

Finally the difference of the rank for $\S$ in $\bD$ and the rank for $\S'$ in $\bD'$ is given by
\begin{align}
&\sum_{S \neq \S} \ell(S) (\chi(S \m \S)-\chi(S \m \S')) - \sum_{W \neq \W} (\chi(W \m \S)-\chi(W \m \S')) \nonumber\\
=&\sum_{S \neq \S} \ell(S) \chi(\S' \m S \m \S) - \sum_{W \neq \W} \chi(\S' \m W \m \S), \nonumber
\end{align}
since $\S' \m \S$ and $\W' \m \W$. Now $\chi(\S'\m S \m\S)=1$ only when $S\in S_{\{i-1\}}^\L$ and $\chi( \W' <^s W < \W)=1$ only when $W \in W^\L_{\{i-1\}} \bigcup W^\R_{\{i\}}$.

By summing over all such $S$ and $W$, the difference becomes
\begin{align}
r(\S)-r({\S'})&=\sum_{S\in S^\L_{\{i-1\}}}\ell(S) - |W^\L_{\{i-1\}}|-|W^\R_{\{i\}}| \nonumber \\
&=\sum_{S\in S^\L_{\{i-1\}}}\ell(S) - c^b(B_1)-c^b(B_2).\label{area-diff-S}
\end{align}
where $c^b(B_1), c^b(B_2)$ denote blue segment counts in the corresponding regions in Figure \ref{fig:the-cell} (left picture).
By Proposition \ref{prop--row-zero},
%in every row of a path diagram , the red segments and blue segments have to alternate,
we have $W^\L_{\{i-1\}}=c^b(B_1)$ and $W^\R_{\{i\}}=c^b(B_2)$.

The proposition then follows by adding the two formulas (\ref{area-diff-noS}) and (\ref{area-diff-S}).
\end{proof}
%\begin{align}\label{rankS-S}
%rank(\S)-rank({\d})=\sum_{S\in S^\L_{\{i-1\}}}\ell(S)+\sum_{S\in S^\R_{\{i\}}}\ell(S)-c^b(B_1)-(c^r(B_2)+1).
%\end{align}

%\begin{figure}[!ht]
%  \begin{center}
% \includegraphics[height= 2.4 in]{areaProof.png}
%  \end{center}
%  \caption{$r(\S)=i$.
%  }
%\end{figure}

%\begin{prop}\label{rec-area}
%Let $\overline{D'}$ be obtained from $\overline{D}$ by removing the area cell, and let $D$ and $D'$ be their sweep map images. Then
%
%$$area(D)-area(D')=
%    -k \times |S_{\{i-1\}}^\L|+|S^\L_{\{i\}}| + \sum_{S\in S^\L_{\{i-1\}}}\ell(S) - c^b(B_1)-c^b(B_2).$$
%    $$=-\sum_{ S \in S_{\{i-1\}}^\L}(k-\ell(S))+|S^\L_{\{i\}}| - c^b(B_1)-c^b(B_2).$$
%
%\end{prop}

\begin{prop}\label{rec-dinv}
Let $\overline{D}'$ be obtained from $\overline{D}$ by removing an area cell. Then
\begin{align}
\dinv(\overline{D})-\dinv(\overline{D}')&=\sum_{S\in S^\L_{\{i-1\}}}\ell(S) - k \times |S_{\{i-1\}}^\L| + |S^\L_{\{i\}}| -c^b(B_1)- c^b(B_2)\nonumber
%= &-\sum_{S\in S^\L_{\{i-1\}}}(k-\ell(S))+ |S^\L_{\{i\}}|-c^r(T_1)-1 - c^r(B_2)+c^b(T_1)-c^b(B_1).
\end{align}
\end{prop}

\begin{proof}
We give the $\dinv$ recursion $\dinv(\overline{D})-\dinv(\overline{D'})$ that can be stated two parts as follows:

Part $1$: The difference for sweep $\dinv$ coming from $(W_i\rightarrow S_j).$
%we will use the visual fact in (\cite{dinv-area}, Proposition $4$).

Since $\overline{D'}$ is obtained from $\overline{D}$ by replacing the solid arrows $\S ,\W$ by dashed arrows $\d ,\W'$, we can divide the contribution of a pair $(W_i \rightarrow S_j)$ to the difference into four cases.

(1) Both $W_i$ and $S_j$ are not in the displayed arrows. The contribution in this case is always $0$.

%\begin{figure}[!ht]
%  \begin{center}
% \includegraphics[height= 1.4 in]{dinvProof.png}
%  \end{center}
%  \caption{Contribution for the dinv difference.
%  \label{fig:the-cell-dinv}}
%\end{figure}

(2) Both $W_i$ and $S_j$ are in the displayed arrows. This can only happen when $(\W, \S)$ in $\overline{D}$ (no $\dinv$) becomes $(\W', \d)$ in $\overline{D'}$ ($1$ $\dinv$). Therefore the contribution to the difference in this case is $-1$.

(3) Only $W_i$ is one of the displayed arrows. This means $(\W, S_j)$ in $\overline{D}$ becomes $(\W', S_j)$ in $\overline{D'}$.
Observe that $\#\{(\W\to S_j)\}= c^r(T_2)$ and $\#\{(\W'\to S_j)\}=c^r(B_2)$.
Therefore the contribution to the difference in this case is
$c^r(T_2) - c^r(B_2)$.

(4) Only $S_j$ is in the displayed arrows. This means $(W_i,\S)$ in $\overline{D}$
becomes $(W_i, \d)$ in $\overline{D'}$.
Their contribution to the difference is $1$ if $W_i$ has a blue segment in $T_1$, $-1$ if $W_i$ has a blue segment in $B_1$ and $0$ if $W_i$ does not have a segment in neither $T_1$ or $B_1$.
Therefore the contribution to the difference in this case is $c^b(T_1)-c^b(B_1)$.

 where $c^b(B_1), c^b(T_1)$ denote blue segment counts and $c^r(B_2), c^r(T_2)$ denote red segment
counts in the corresponding regions in Figure \ref{fig:the-cell} (right picture).

So the total contribution to the difference in this part is
\begin{align}
-1 + c^r(T_2) - c^r(B_2)+c^b(T_1)-c^b(B_1).\label{dinv-diff-sweep}
\end{align}

Part 2: For red $\dinv$, we need to consider three cases.

Case 1: $\S, \S'$ are not involved. Then $(S_i,S_j)$ in $\bD$ becomes $(S_i,S_j)$ in $\bD'$ and the $\dinv$ difference is $0$.

Case 2: $(S_i,\S)$ in $\bD$ becomes $(S_i,\S')$ in $\bD'$. The red $\dinv$ for this type in $\bD$ is given by
\begin{multline*}
\sum_{S<\S} \chi(r(S) \geq r(\S) \;\&\; \r(\S) > \r(S))  (\r(\S)- \r(S)) \\
+ \sum_{S<\S}\chi(r(S) < r(\S) \;\&\; \r(S) > \r(\S)) (\r(S)- \r(\S)).
\end{multline*}
Recall that by our choice of $\S$, $r(S)>r(\S)=i$ is impossible. Thus the sum becomes
\begin{align*}
 &\sum_{S\in S^\L_{\{i\}} / S^{\L,i+k}_{\{i\}}} (\r(\S)- \r(S)) +\sum_{S\in S^{\L,i+k+1}_{\{0,1,..,i-1\}}} (\r(S)- \r(\S)).\\
 =&\sum_{S\in S^\L_{\{i\}} / S^{\L,i+k}_{\{i\}}} (i+k- \r(S)) +\sum_{S\in S^{\L,i+k}_{\{0,1,..,i-1\}}} (\r(S)- k-i),
\end{align*}
where we have add $0=(\r(S)-k-i)$ for those $S$ with $\r(S)=k+i$.

The red $\dinv$ for this type in $\bD'$ is similar:
\begin{align*}
&\sum_{S\in S^\L_{\{i,i-1\}} / S^{\L,i+k-1}_{\{i,i-1\}}} (k+i-1- \r(S)) +\sum_{S\in S^{\L,i+k}_{\{0,1,..,i-2\}}} (\r(S)- k-i+1) \\
=&\sum_{S\in S^\L_{\{i,i-1\}} / S^{\L,i+k}_{\{i,i-1\}}} (k+i-1- \r(S)) +\sum_{S\in S^{\L,i+k}_{\{0,1,..,i-2\}}} (\r(S)- k-i+1),
\end{align*}
where we have add $0=k+i-1-\r(S)$ for those $S$ with $\r(S)=k+i-1$.

Their difference is given by
\begin{align*}
&-\!\!\!\!\!\!\sum_{S\in S^\L_{\{i-1\}}/S^{\L,i+k}_{\{i-1\}}}\!\!\!(k-\ell(S))+ |S^\L_{\{i\}}/S^{\L, i+k}_{\{i\}}|-|S^{\L, i+k}_{\{0,1,..,i-2\}}|+\sum_{S\in S^{\L,i+k}_{\{i-1\}}} (\r(S)- k-i)\\
=&\!\!\!\!\sum_{S\in S^\L_{\{i-1\}}}\!\!\!(\ell(S)-k)-\!\!\!\!\!\sum_{S\in S^{\L,i+k}_{\{i-1\}}}\!\!\!\!(\ell(S)-k)+ |S^\L_{\{i\}}/S^{\L, i+k}_{\{i\}}|-|S^{\L, i+k}_{\{0,1,..,i-2\}}|+\!\!\!\!\sum_{S\in S^{\L,i+k}_{\{i-1\}}}\!\!\!\!(\ell(S)- k-1)\\
=&\!\!\!\!\!\sum_{S\in S^\L_{\{i-1\}}}(\ell(S)-k)+ |S^\L_{\{i\}}/S^{\L, i+k}_{\{i\}}|-|S^{\L, i+k}_{\{0,1,..,i-2\}}|- |S^{\L,i+k}_{\{i-1\}}|
\end{align*}

Case 3: $(\S,S_j)$ becomes $(\S',S_j)$. The red $\dinv$ in $\bD$ is
$$\sum_{\S <S} \chi(r(\S) \geq r(S) \;\&\; \r(S) > \r(\S)) (\r(S)- \r(\S))$$

since $r(\S)<r(S)$ is impossible by our choice of $\S$. Thus the sum becomes
 $$\sum_{S\in S^{\R,i+k+1}_{\{0,1,..,i-1\}} }(\r(S)-(i+k))=\sum_{S\in S^{\R,i+k}_{\{0,1,..,i-1\}} }(\r(S)-(i+k)).$$

The red $\dinv$ in $\bD'$ is similar:
$$\sum_{\d <S} \chi(r(\d) \geq r(S) \;\&\; \r(S) > \r(\d))(\r(S)- \r(\d))=\!\!\!\!\!\sum_{S\in S^{\R,i+k}_{\{0,1,..,i-1\}} }\!\!\!\!(\r(S)-(i+k-1)).$$
since $r(\d)<r(S) \Rightarrow r(\S)\le r(S)$ is impossible by our choice of $\S$.

Their difference is
$$-|S^{\R,i+k}_{\{0,1,..,i-1\}}|.$$

So the contribution to the difference in this part is
\begin{align}
&\sum_{S\in S^\L_{\{i-1\}}}(\ell(S)-k)+ |S^\L_{\{i\}}/S^{\L, i+k}_{\{i\}}|-|S^{\L, i+k}_{\{0,1,..,i-2\}}|- |S^{\L,i+k}_{\{i-1\}}|-|S^{\R,i+k}_{\{0,1,..,i-1\}}| \nonumber\\
=&\sum_{S\in S^\L_{\{i-1\}}}(\ell(S)-k)+ |S^\L_{\{i\}}|-|S^{\L, i+k}_{\{i\}}|-|S^{\L, i+k}_{\{0,1,..,i-2\}}|- |S^{\L,i+k}_{\{i-1\}}|-|S^{\R,i+k}_{\{0,1,..,i-1\}}| \nonumber\\
=&\sum_{S\in S^\L_{\{i-1\}}}(\ell(S)-k)+ |S^\L_{\{i\}}|-|S^{\L, i+k}_{\{0,1,..,i\}}|-|S^{\R,i+k}_{\{0,1,..,i-1\}}| \nonumber \\
=&\sum_{S\in S^\L_{\{i-1\}}}\ell(S) - k \times |S_{\{i-1\}}^\L|+ |S^\L_{\{i\}}|-c^r(T_1)-c^r(T_2). \label{dinv-diff-red}
\end{align}
where $c^r(T_1), c^r(T_2)$ denote red segment counts in the corresponding regions in Figure \ref{fig:the-cell} (left picture).
Recall that by our choice of $\S$, we have $c^r(T_1) = |S^{\L, i+k}_{\{0,1,..,i\}}|$ and $c^r(T_2) = |S^{\R,i+k}_{\{0,1,..,i-1\}}|$.

The formula (\ref{dinv-diff-sweep}) is
$$-1 + c^r(T_2) - c^r(B_2)+c^b(T_1)-c^b(B_1).$$

The proposition then follows by adding the two formulas (\ref{dinv-diff-sweep}) and (\ref{dinv-diff-red}), and using the
fact $c^b(T_1) = c^r(T_1)$ and $c^r(B_2)+ 1 = c^b(B_2) $, which are consequences of Proposition \ref{prop--row-zero}.
\end{proof}

%$Proof~that~dinv~sweeps~to~area.$ We first show that the recursions in Proposition \ref{rec-area} and \ref{rec-dinv} are identical.
%To do this it is sufficient to observe that $c^r(T_1) = c^b(T_1)$ and $c^r(B_2) + 1 = c^b(B_2)$.
%The reason for this is the alternating colors of segments in each row that always begin
%with a red segment and end with a blue segment.

%Thus it is sufficient to verify the base case where $dinv(\overline{D})=0$. The $\vec{k}$-Dyck path with dinv $0$ is the path $\overline{D}$ that
%red arrows $S^{k_{1}},S^{k_{2}},...,S^{k_{n}}$  and then  $|\k|$ blue arrow $W$.
%This forces the image $D$ of $\overline{D}$ to a red arrow $S^{k_1}$ followed by $k_1$ blue arrow $W$, then a red arrow $S^{k_2}$ followed by $k_2$ blue arrow $W$, and so on. Thus the starting ranks of all north steps are zero.

\section{Proof the area sweeps to bounce\label{s-proof-areatobounce}}
Our proof relies on the inverting sweep map in \cite{Xin-Zhang}.
We will quote some results for the readers' convenience.

\begin{algor}[Filling Algorithm \cite{Xin-Zhang}]
%\label{al-Filling Algorithm}

\noindent
Input: The SW-sequence $\SW(D)$ of a $\k$-Dyck path $D\in \CD_{\k}$.

\noindent
Output: A tableau $T=T(D)\in \TAU_{\k}$.

\begin{enumerate}
\item   Start by placing a $1$ in the top row and the first column.

\item  If the second letter in $\SW(D)$ is an $S^*$ we put a $2$ on the top of the second column.

\item   If the second letter in $\SW(D)$ is a $W$ we place $2$ below the $1$.

\item  At any stage the entry at the bottom of the $i$-th column but not in row $k_i+1$ will be called \emph{active}.

\item  Having placed $1,2,\cdots i-1$, we place $i$ immediately below the smallest  active entry if the $i^{th}$ letter in $\SW(D)$ is a $W$, otherwise we place $i$ at the top of the first empty column.

\item   We carry this out recursively until $1,2,\ldots ,n+|\k|$ have  all  been placed.
\end{enumerate}
\end{algor}

\begin{algor}[Ranking Algorithm \cite{Xin-Zhang}]
%\label{al-Ranking Algorithm}

\noindent
Input: A tableau $T=T(D )\in \TAU_{\k}$.
%The SW-sequence $\SW(D)$ of a Dyck path $D\in \CD_{m,n}$ where $m=kn$.

\noindent
Output: A rank tableau $R(D)$ of the same shape with $T$.

\begin{enumerate}
\item  Successively assign $0, 1, 2, ..., k_1$ to the first column indices of $T$ from top to bottom;

\item For $i$ from $2$ to $n$, if the top index of the $i$-th column is $A+1$, and the rank of index $A$ is $a$, then assign the index $A+1$ rank $a$. Moreover,
the ranks in the $i$-th column are successively $a,a+1,\dots, a+k_i$ from top to bottom.

\end{enumerate}
\end{algor}

 For instance if $D $ is the path in Figure \ref{fig:bouncepath}, with SW-sequence
 \begin{align*}
  \SW(D)=
  \begin{array}
   {ccccccccccccccc}
  S^4 \!\!\!\!\!& S^3 \!\!\!\!\! & W \!\!\!\!\! & W \!\!\!\!\! & W\!\! \!\!\!& S^1 \!\!\!\!\! &W\!\! \!\!\!&S^1\!\!\!\!\! & W \!\!\!\!\!& S^1 \!\!\!\!\! & W \!\!\!\!\! & W \!\!\!\!\! & W \!\!\!\!\! & W \!\!\!\!\!&W
 \end{array},
\end{align*}
then we obtain the  tableau $T(D)$ and $R(D)$ in Figure \ref{fig:filingAl}.
\begin{figure}[!ht]
  $$\hskip .8 in T(D) = \hskip -2.4 in  \vcenter{ \includegraphics[height=1 in]{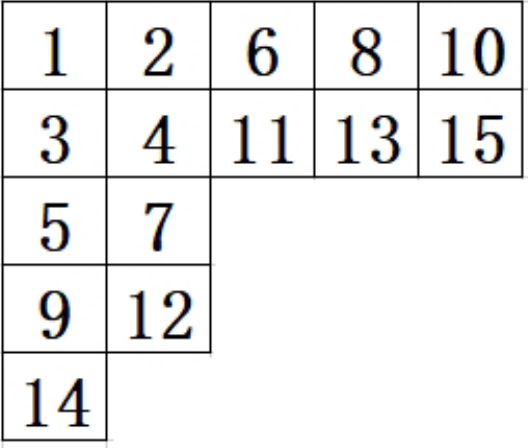}}
  \hskip -1.9in R(D)= \hskip -2.4 in \vcenter{ \includegraphics[height=1 in]{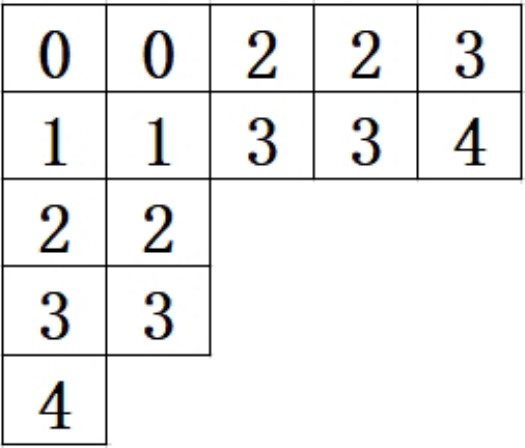}}
  $$
\caption{The filling tableau $T(D)$ and the rank tableau $R(D)$ of the path in Figure \ref{fig:bouncepath}.
\label{fig:filingAl}}
\end{figure}

The following result is a summary of Lemmas 3.1, 3.2 and Theorem 2.14 in \cite{Xin-Zhang}.
\begin{theo}
For a Dyck path $D\in D_{\k}$, Let $\bD$ be the preimage of $D$ on the sweep map. We obtain a Filling tableau $T(D)$
and a Ranking tableau $R(D)$ by Filling algorithm and Ranking algorithm. The Ranking algorithm assigns every index
a rank in $T(D)$ and the ranks are weakly increasing according to their indices. If indices $1,2,\dots,n+|\k|$
are assigned ranks $r_1,r_2,\dots,r_{n+|\k|}$, then the rank sequence of $\bD$ is exactly $(r_1,r_2,\dots,r_{n+|\k|})$.
\end{theo}

Now we are ready to prove that the $\area$ sweeps to $\bounce$.
\begin{proof}[Proof of Theorem \ref{dinv-area-bounce1} part 2]
Note that $\area(\bD)=\bounce(D)$, and the ranks of the south ends of $\bD$ are just the first row entries of $R(D)$. It suffices to show that
the tableau $R^b(D)$ is the same as the ranking tableau $R(D)$.

Since both tableaux have columns of the form $a,a+1,a+2,\dots$ from top to bottom,
and have the first row weakly increasing, it suffices to show the following claim.

Claim: $R(D)$ has $v_i$ $i$'s for $i=0,1,\dots, $ in its first row.

We prove the claim by induction on $i$.

The base case is when $i=0$. By definition,
$D$ starts with $v_0$ north steps followed by an east step.
Now the filling algorithm will produce
$1,2,\dots, v_0$ in the first row, with $v_0+1$ under $1$. It follows that $R(D)$ has $v_0$ $0$'s in the first row, and has only $v_0$ $0$' since the rank of $v_0+1$ is already $1$.
The claim then holds true in this case.

Assume the claim holds true for $i$ and we need to show the case $i+1$.

Consider the bouncing path part $P_{i}$ goes north to $Q_{i}$ (in $D$), goes east to $P_{i+1}$, and goes north to $Q_{i+1}$ (in $D$).
We have the following facts:

i) From $Q_{i}$ to $Q_{i+1}$ in $D$, there are $h_{i}$ east steps  and $v_{i+1}$ north steps, with indices $\sum_{j=0}^{i-1} (v_j+h_j) + v_i + s$
for $s=1,2,\dots, v_{i+1}+h_i$.

ii) Suppose now we have filled the indices up to $Q_i$ in the filling tableaux.
By the induction hypothesis, $R^i$ agrees with $R(D)$, so that these indices corresponds to the ranks no more than $i$ below the first row of $R^i$, with
the number of $j+1$'s being equal to $h_{j}$ for $j\le i-1$. By the filling algorithm, the bottom indices are active only when its rank
is $i$ and the cell under it has rank $i+1$ in $R^i$. This implies, the index $1+\sum_{j=0}^{i-1} (v_j+h_j)+v_i$, corresponding to the west end $Q_i$,
must be filled under one of the active $i$'s. Consequently, the indices of the $h_i$ east steps in the path from $Q_i$ to $Q_{i+1}$ must be filled in the $h_i$ cells
of rank $i+1$ in $R^i$, and the indices of the $v_{i+1}$ north steps in the path from $Q_i$ to $Q_{i+1}$ must also have rank $i+1$ by the ranking algorithm. To see that there is no more
rank $i+1$ for north steps, we observe that the next index corresponds to the west end $Q_{i+1}$. By the filling algorithm, this index cannot be put below an index of rank less than $i+1$,
and hence must have rank $i+2$.
\end{proof}

\section{Summary\label{s-summary}}
We have defined the $\dinv$, $\area$, and $\bounce$ statistics for $\k$-Dyck paths, and proved that the sweep map takes $\dinv$ to $\area$, and $\area$ to $\bounce$.
Such a result was first known by  Haiman and Haglund for classical (or ordinary) Dyck paths;
The result was extended for $k$-Dyck paths by Loehr. Our result includes the two mentioned cases as special cases.

The $\dinv$ sweeps to $\area$ result was also known for $(m,n)$ rational Dyck paths by \cite{Loehr-Warrington,Gorsky-Mazin,Mazin,dinv-area}. Our work for $\k$-Dyck paths are inspired by Garsia-Xin's visual proof in \cite{dinv-area}.
We should mention that such a result also has a parking function version for ordinary Dyck paths. See, e.g., \cite{Hag-book08}.
Finding a proper extension of $\k$-Dyck paths to $\k$-parking functions is one of our future projects.

The $\bounce$ statistic was only known for classical Dyck paths, $k$-Dyck paths, and remains unknown for rational Dyck paths.

We also investigated the $q,t$-symmetry of $C_{\lambda}(q,t)$. The symmetry is easily proved when the length of $\lambda$ is $\ell(\lambda)=2$,
and hopefully will be proved for the case $\ell(\lambda)=3$ in an upcoming paper.
The symmetry no longer holds in general for $n\ge 4$. But computer experiments suggest the following conjecture.

\begin{conj}
Let $\lambda=((a+1)^s,a^{n-s})$ be consisting of $s$ copy of $a+1$'s  and $n-s$ copy of $a$'s. Then $C_{\lambda}(q,t)$ is $q,t$-symmetric.
\end{conj}
The $s=0$ and $s=n$ cases reduce to $k$-Dyck paths, and the conjecture holds true in these cases. %We have checked the conjecture for $a=1$ and $n\le $, $a=2$ and $n\le $.

\end{document}